\documentclass[12pt,twoside]{amsart}
\usepackage{amssymb,amscd,amsxtra,calc,mathrsfs}
\usepackage{amsmath}
\usepackage{xypic}

\title[Fano threefolds of degree 22]{On Fano threefolds of degree $22$ after 
Cheltsov and Shramov}
\author{Kento Fujita} 
\date{\today}
\address{Department of Mathematics, Graduate School of Science, Osaka University, 
Toyonaka, Osaka 560-0043, Japan}
\email{fujita@math.sci.osaka-u.ac.jp}

\newcommand{\pr}{\mathbb{P}}

\newcommand{\Z}{\mathbb{Z}}
\newcommand{\Q}{\mathbb{Q}}
\newcommand{\R}{\mathbb{R}}
\newcommand{\C}{\mathbb{C}}

\newcommand{\G}{\mathbb{G}}

\newcommand{\Supp}{\operatorname{Supp}}

\newcommand{\Pic}{\operatorname{Pic}}

\newcommand{\rank}{\operatorname{rank}}

\newcommand{\Aut}{\operatorname{Aut}}

\newcommand{\lct}{\operatorname{lct}}

\newcommand{\ord}{\operatorname{ord}}

\newcommand{\vol}{\operatorname{vol}}

\newcommand{\PGL}{\operatorname{PGL}}
\newcommand{\SL}{\operatorname{SL}}

\newcommand{\MU}{\operatorname{MU}}
\newcommand{\Nklt}{\operatorname{Nklt}}
\newcommand{\aaa}{\operatorname{a}}

\newcommand{\sC}{\mathcal{C}}

\newcommand{\sO}{\mathcal{O}}

\newcommand{\sD}{\mathcal{D}}

\newtheorem{thm}{Theorem}[section]
\newtheorem{lemma}[thm]{Lemma}
\newtheorem{proposition}[thm]{Proposition}
\newtheorem{corollary}[thm]{Corollary}
\newtheorem{claim}[thm]{Claim}

\theoremstyle{definition}
\newtheorem{definition}[thm]{Definition}

\newtheorem*{ack}{Acknowledgments}

\begin{document}

\maketitle 

\begin{abstract}
It has been known that nonsingular Fano threefolds of Picard rank one with 
the anti-canonical degree $22$ 
admitting faithful actions of the multiplicative group form a one-dimensional family. 
Cheltsov and Shramov showed that all but two of them admit K\"ahler-Einstein metrics.  
In this paper, we show that the remaining Fano threefolds also admit 
K\"ahler-Einstein metrics. 
\end{abstract}

\setcounter{tocdepth}{1}
\tableofcontents

\section{Introduction}\label{intro_section}

Let $X$ be a nonsingular Fano threefold over the complex number field with 
$\Pic(X)=\Z[-K_X]$. Such $X$ with large 
automorphism groups have been studied by many authors. For example, 
Mukai and Umemura systematically studied in \cite{MU} the Fano threefold $V^{\MU}$ with 
the anti-canonical degree $22$, so-called 
\emph{the Mukai-Umemura threefold}, which is obtained by the unique 
$\SL(2, \C)$-equivariant nonsingular projective 
compactification of $\SL(2, \C)/I_h$ with the Picard rank one, 
where $I_h\subset\SL(2, \C)$ is the icosahedral group. 
The automorphism group of $V^{\MU}$ is equal to $\PGL(2, \C)$. 
On the other hand, Prokhorov showed in 
\cite{Prok} that, if $\Aut(X)$ is not finite, then the anti-canonical degree 
of $X$ must be equal to $22$. Moreover, he determined all of such $X$. 
For example, there is a unique Fano threefold $V^{\aaa}$ such that 
$\Aut^0(V^{\aaa})$ is equal to the additive group $\C^+$. 

Nowadays, the structures of such $X$ are well-understood 
thanks to the works of Kuznetsov, Prokhorov, Shramov \cite{KPS}, and Kuznetzov, 
Prokhorov \cite{KP}. The family of nonsingular Fano threefolds $X$ with $\Pic(X)=\Z[-K_X]$, 
$(-K_X)^{\cdot 3}=22$, $X\not\simeq V^{\aaa}$ and $\Aut(X)$ infinite 
are parametrized by 
$\C\setminus\{0,1\}$. Let us denote the family by $\{V_u\}_{u\in\C\setminus\{0,1\}}$. 
Then $V_{-1/4}$ is isomorphic to $V^{\MU}$, and the automorphism group of 
$V_u$ is equal to $G:=\C^*\rtimes(\Z/2\Z)$ unless $u=-1/4$. 
Moreover, any $V_u$ can be obtained by the two-ray game from the blowup of 
some $3$-dimensional quadric hypersurface along a 
certain nonsingular sextic rational curve. 

In \cite{CS}, Cheltsov and Shramov considered the problem for the existence of 
K\"ahler-Einstein metrics for the above $V_u$. If $u=-1/4$, then the Fano threefold 
is the Mukai-Umemura threefold $V^{\MU}$. In this case, Donaldson already 
showed in \cite{don-alpha} the existence of K\"ahler-Einstein metrics on $V^{\MU}$
by showing that the $\PGL(2,\C)$-invariant \emph{$\alpha$-invariant} 
$\alpha_{\PGL(2,\C)}(V^{\MU})$ (see \S \ref{Fano_section}) of $V^{\MU}$ 
is equal to $5/6$. 
In fact, Tian showed in \cite{Tia87} that, for a Fano manifold $X$ and a reductive subgroup 
$\Gamma\subset\Aut(X)$, if $\alpha_\Gamma(X)>\dim X/(\dim X+1)$ holds, then 
$X$ admits K\"ahler-Einstein metrics. Cheltsov and Shramov considered the 
remaining cases by evaluating the $G$-invariant $\alpha$-invariant of $V_u$. 
More precisely, they showed the following: 

\begin{thm}[{\cite[Theorem 1.5]{CS}}]\label{CS_thm}
We have 
\[
\alpha_G(V_u)=\begin{cases}
4/5 & \text{if }u\neq 3/4 \text{ and }u\neq 2, \\
3/4 & \text{if }u=3/4, \\
2/3 & \text{if }u=2
.\end{cases}
\]
\end{thm}

Thus, together with Tian's result, if $u\neq 3/4$ and $u\neq 2$, then $V_u$ admits 
K\"ahler-Einstein metrics. However, if $u=3/4$ or $u=2$, then the existence of 
K\"ahler-Einstein metrics of $V_u$ was not known at that time. 
The purpose of this paper is to show the existence of K\"ahler-Einstein metrics 
for $V_{3/4}$ and $V_2$. The main result is the following: 

\begin{thm}\label{main_thm}
Both $V_{3/4}$ and $V_2$ admit K\"ahler-Einstein metrics. 
\end{thm}

Thus, together with Matsushima's obstruction \cite{matsushima}, we gave a complete 
answer for the existence of K\"ahler-Einstein metrics of nonsingular Fano threefolds $X$ 
with $\Pic(X)=\Z[-K_X]$ and $\Aut(X)$ infinite; 
such $X$ admits K\"ahler-Einstein metrics if and only if $X\not\simeq V^{\aaa}$.

The main technique to prove Theorem \ref{main_thm} is the ``$G$-valuative criterion" 
(see \S \ref{Fano_section}, \cite{FANO} and \cite{zhuang}), 
which is a $G$-equivariant version of 
\cite[Theorem 3.7]{li} and \cite[Theorem 1.6]{vst}. Moreover, we need deep analyses 
\cite{CS} of 
$G$-invariant curves on $V_u$ which might be possible destabilizing centers of 
$G$-invariant prime divisors over $V_u$. 
Technically, the theory of quasi-log schemes \cite{fujino} and the careful analysis 
of the volume functions (see \S \ref{volume_section}) play important roles in order to 
show Theorem \ref{main_thm}. Especially, we use 
a subadjunction-type result for projective qlc strata on quasi-log schemes \cite{fujino_hyp} 
(see Theorem \ref{subadjunction_thm}). 
We expect that 
these techniques will be applied for many 
other Fano varieties (cf. \cite{FANO}).

\begin{ack}
The author thank Carolina Araujo, Ana-Maria Castravet, Ivan Cheltsov, Osamu Fujino, 
Anne-Sophie Kaloghiros, Akihiro Kanemitsu, Jesus Martinez-Garcia, Constantin Shramov, 
Hendrick S\"uss and Nivedita Viswanathan 
for many discussions. 
This work started while the author participated the AIM workshop named 
``K-stability and related topics" on January 2020. 
The author thanks the staffs for the stimulating environment. 
This work was supported by JSPS KAKENHI Grant Number 18K13388.
\end{ack}

\section{Log Fano pairs}\label{Fano_section}

In this section, we recall the definition of $\alpha$-invariant for log Fano pairs and 
see the relationship between $\alpha$-invariant and the existence of 
K\"ahler-Einstein metrics. For the minimal model program, we refer the readers to 
\cite{KoMo}. 

\begin{definition}\label{fano_definition}
A \emph{log Fano pair} $(X, \Delta)$ is a pair of a complex normal projective variety 
$X$ and an effective $\Q$-Weil divisor $\Delta$ on $X$ such that 
the pair $(X, \Delta)$ is klt 
and $-(K_X+\Delta)$ is an ample $\Q$-divisor. If $X$ is nonsingular and $\Delta=0$, 
then $X$ is said to be a \emph{Fano manifold}. 
\end{definition}

\begin{definition}[{$\alpha$-invariant (see \cite{Dem08} 
for example)}]\label{alpha_definition}
Let $(X, \Delta)$ be a log Fano pair and let $\Gamma\subset\Aut(X, \Delta)$ be an 
algebraic subgroup. 
\begin{enumerate}
\renewcommand{\theenumi}{\arabic{enumi}}
\renewcommand{\labelenumi}{(\theenumi)}
\item\label{alpha_definition1}
The \emph{$\Gamma$-invariant $\alpha$-invariant $\alpha_\Gamma(X, \Delta)$ 
of $(X, \Delta)$} is defined as the supremum of $\alpha\in\Q_{>0}$ such that 
the pair $\left(X, \Delta+\frac{\alpha}{m}\sD\right)$ is lc for any $m\in\Z_{>0}$ with 
$-m(K_X+\Delta)$ Cartier and for any nonempty $\Gamma$-invariant 
sub-linear system $\sD\subset|-m(K_X+\Delta)|$.
\item\label{alpha_definition2}
For any scheme-theoretic point $\eta\in X$, we define $\alpha_{\Gamma,\eta}(X,\Delta)$
as the supremum of $\alpha\in\Q_{>0}$ such that 
the pair $\left(X, \Delta+\frac{\alpha}{m}\sD\right)$ is 
lc at $\eta$ for any $m\in\Z_{>0}$ with 
$-m(K_X+\Delta)$ Cartier and for any nonempty $\Gamma$-invariant 
sub-linear system $\sD\subset|-m(K_X+\Delta)|$. 
\item\label{alpha_definition3}
If $\Delta=0$, then we simply write $\alpha_\Gamma(X)$ and $\alpha_{\Gamma,\eta}(X)$ 
in place of $\alpha_\Gamma(X,0)$ and $\alpha_{\Gamma,\eta}(X,0)$, respectively. 
\end{enumerate}
\end{definition}

\begin{lemma}\label{solvable_lemma}
Let $(X, \Delta)$ be a log Fano pair, $\Gamma\subset\Aut(X, \Delta)$ be an algebraic 
subgroup, and $\eta\in X$ be a scheme-theoretic point. 
Assume that the identity component $\Gamma_0$ of $\Gamma$ is solvable. 
\begin{enumerate}
\renewcommand{\theenumi}{\arabic{enumi}}
\renewcommand{\labelenumi}{(\theenumi)}
\item\label{solvable_lemma1}
The value $\alpha_\Gamma(X, \Delta)$ is equal to the supremum of $\alpha\in\Q_{>0}$ 
such that the pair $(X, \Delta+\alpha D)$ is lc for any effective $\Gamma$-invariant 
$\Q$-divisor $D\sim_\Q -(K_X+\Delta)$.
\item\label{solvable_lemma2}
The value $\alpha_{\Gamma,\eta}(X, \Delta)$ is equal to the supremum 
of $\alpha\in\Q_{>0}$ such that the pair $(X, \Delta+\alpha D)$ is lc at $\eta$ 
for any effective $\Gamma$-invariant 
$\Q$-divisor $D\sim_\Q -(K_X+\Delta)$.
\end{enumerate}
\end{lemma}

\begin{proof}
We only prove \eqref{solvable_lemma2}. By \cite[Lemme 5.11]{borel-serre}, there is 
a finite algebraic subgroup $\Gamma_1\subset\Gamma$ such that 
$\Gamma_1$ meets every connected component of $\Gamma$. 
Set $d:=\#\Gamma_1$. Fix $\alpha\in\Q_{>0}$. 

Assume that $\left(X, \Delta+\frac{\alpha}{m}\sD\right)$ is lc at $\eta$ for any $m$ and 
$\sD\subset|-m(K_X+\Delta)|$. Then, for any effective $\Gamma$-invariant $\Q$-divisor 
$D\sim_\Q-(K_X+\Delta)$, since $\{mD\}\subset|-m(K_X+\Delta)|$ is a 
$\Gamma$-invariant sub-linear system for $m$ sufficiently divisible, we know that 
the pair $\left(X, \Delta+\frac{\alpha}{m}\{mD\}\right)$ is lc at $\eta$, i.e., the pair 
$(X, \Delta+\alpha D)$ is lc at $\eta$. 

Conversely, assume that $(X, \Delta+\alpha D)$ is lc at $\eta$ for any effective 
$\Gamma$-invariant $\Q$-divisor $D\sim_\Q-(K_X+\Delta)$. Take any 
$\Gamma$-invariant sub-linear system $\sD\subset|-m(K_X+\Delta)|$. Since 
$\Gamma_0$ is connected and solvable, there exists a $\Gamma_0$-invariant divisor 
$D_0\in \sD$ by the Borel fixed point theorem \cite[\S 21.2]{humphreys}. Let us set 
\[
\tilde{D}:=\sum_{h\in\Gamma_1}h(D_0)\in d\sD. 
\]
Since $h(D_0)\in\sD$ is $\Gamma_0$-invariant for any $h\in\Gamma_1$, 
the divisor $\tilde{D}\in d\sD$ is $\Gamma$-invariant. Set 
$D:=\frac{1}{md}\tilde{D}\sim_\Q-(K_X+\Delta)$. 
Since $(X, \Delta+\alpha D)$ is lc at $\eta$, 
the pair $\left(X, \Delta+\frac{\alpha}{md}(d\sD)\right)$ 
is also lc at $\eta$. This is equivalent to 
the pair $\left(X, \Delta+\frac{\alpha}{m}\sD\right)$ being lc at $\eta$. 
\end{proof}

We recall the words on \emph{K-stability} of log Fano pairs. The original notion of 
K-stability was introduced by Tian \cite{tian} and Donaldson \cite{don} by using the 
languages of \emph{test configurations}. In this paper, we only treat its simplification 
due to Li \cite{li} and the author \cite{vst}. 

\begin{definition}\label{vst_definition}
Let $(X, \Delta)$ be an $n$-dimensional log Fano pair and let $F$ be a prime divisor 
over $X$ obtained by a log resolution $\pi\colon\tilde{X}\to X$ of $(X, \Delta)$ 
(that is, $F$ is a prime divisor on $\tilde{X}$). 
\begin{enumerate}
\renewcommand{\theenumi}{\arabic{enumi}}
\renewcommand{\labelenumi}{(\theenumi)}
\item\label{vst_definition1}
Let $A_{X,\Delta}(F)$ be the log discrepancy of $(X, \Delta)$ along $F$, that is, 
$1$ plus the coefficient of $K_{\tilde{X}}-\pi^*(K_X+\Delta)$ along $F$. 
\item\label{vst_definition2}
For any effective $\Q$-Cartier $\Q$-divisor $D$, let 
$\ord_F D\in\Q_{\geq 0}$ be the coefficient of $\pi^*D$ along $F$. 
For any $m\in\Z_{\geq 0}$ with $-m(K_X+\Delta)$ Cartier and for any $j\in\R_{\geq 0}$, 
let 
\[
H^0(X, -m(K_X+\Delta)-jF)\subset H^0(X, -m(K_X+\Delta))
\]
be the sub-vector space corresponds to the sub-linear system 
$|-m(K_X+\Delta)-jF|\subset|-m(K_X+\Delta)|$ 
consisting all $D\in|-m(K_X+\Delta)|$ with $\ord_F D\geq j$.
\item\label{vst_definition3}
For any $x\in\R_{\geq 0}$, let us set 
\[
\vol(-(K_X+\Delta)-xF):=\lim_{m\to\infty}\frac{
\dim H^0(X, -m(K_X+\Delta)-mxF)}{m^n/n!},
\]
where $m$ runs through all positive integers with $-m(K_X+\Delta)$ Cartier 
(the limit exists by \cite{L1,L2}). 
Obviously, $\vol(-(K_X+\Delta)-0\cdot F)=(-(K_X+\Delta))^{\cdot n}$ holds. 
It follows from the definition that 
\[
\vol(-(K_X+\Delta)-xF)=\vol_{\tilde{X}}(-\pi^*(K_X+\Delta)-xF).
\]
In particular, by \cite{L1, L2}, the function 
$\vol(-(K_X+\Delta)-xF)$ is a non-increasing and continuous function over $x\in[0,\infty)$. 
Moreover, if $x\gg 0$, then $\vol(-(K_X+\Delta)-xF)=0$ holds. Let us set 
\[
\tau_{X, \Delta}(F):=\sup\{\tau\in\R_{> 0}\,\,|\,\,\vol(-(K_X+\Delta)-\tau F)>0\}.
\]
\item\label{vst_definition4}
Let us set 
\[
S_{X, \Delta}(F):=\frac{1}{(-(K_X+\Delta))^{\cdot n}}\int_0^\infty\vol(-(K_X+\Delta)-xF)dx.
\]
\end{enumerate}
The above definitions do not depend on the choice of $\pi$. 
We often write $A(F)$, $\tau(F)$, $S(F)$ in place of $A_{X,\Delta}(F)$, 
$\tau_{X, \Delta}(F)$, $S_{X, \Delta}(F)$, just for simplicity. 
\end{definition}

The following lemma is well-known: 

\begin{lemma}\label{alpha-tau_lemma}
Let $(X, \Delta)$ be a log Fano pair, let $\Gamma\subset\Aut(X, \Delta)$ be an algebraic 
subgroup, and let $\eta\in X$ be a scheme-theoretic point. For any $\Gamma$-invariant 
prime divisor $F$ over $X$ with $\eta\in c_X(F)$, we have 
\[
\frac{A(F)}{\tau(F)}\geq \alpha_{\Gamma,\eta}(X,\Delta),
\]
where $c_X(F)$ is the center of $F$ on $X$. 
\end{lemma}

\begin{proof}
Take any $\tau=\frac{j}{m}\in(0,\tau(F))\cap\Q$. Then the sub-linear system 
\[
|-m(K_X+\Delta)-jF|\subset|-m(K_X+\Delta)|
\]
is nonempty, $\Gamma$-invariant, and vanishes along $F$ at least $j$ times. 
Thus we have 
\begin{eqnarray*}
\alpha_{\Gamma,\eta}(X, \Delta)&\leq& 
\lct_\eta\left(X,\Delta;\,\, \frac{1}{m}|-m(K_X+\Delta)-jF|\right)\\
&\leq&\frac{A_{X, \Delta}(F)}{\ord_F\left(\frac{1}{m}|-m(K_X+\Delta)-jF|\right)}
\leq\frac{A_{X, \Delta}(F)}{\tau}, 
\end{eqnarray*}
where $\lct_\eta$ is the log canonical threshold at $\eta$. 
\end{proof}

We see a $G$-invariant version of \cite{li, vst} and \cite{alpha}. See also \cite{zhuang} 
for more general frameworks.

\begin{proposition}[{see \cite{FANO}}]\label{G_proposition}
Let $X$ be an $n$-dimensional Fano manifold and let $\Gamma\subset\Aut(X)$ 
be a reductive subgroup. 
\begin{enumerate}
\renewcommand{\theenumi}{\arabic{enumi}}
\renewcommand{\labelenumi}{(\theenumi)}
\item\label{G_proposition1} (see also \cite[Corollary 4.13]{zhuang})
Assume that $A(F)>S(F)$ holds for any $\Gamma$-invariant prime divisor over $X$ with 
the $\C$-algebra 
\[
\bigoplus_{m,j\in\Z_{\geq 0}}H^0(X, -mK_X-jF)
\]
finitely generated. Then $X$ admits K\"ahler-Einstein metrics. 
\item\label{G_proposition2} (see also \cite{Tia87} and \cite[Theorem 1.3]{alpha})
If $\alpha_\Gamma(X)\geq\frac{n}{n+1}$, then $X$ admits K\"ahler-Einstein metrics. 
\end{enumerate}
\end{proposition}

For the complete proof, see \cite{FANO}. We only give a sketch of the idea. 
For \eqref{G_proposition1}, for any $\Gamma$-equivariant 
special degeneration of $X$ in the sense of 
\cite{DS}, the corresponding prime divisor $F$ over $X$ in \cite[Theorem 5.1]{vst} 
obviously satisfies the assumptions in \eqref{G_proposition1}. By \cite[Theorem 5.1]{vst}, 
the signatures of the Donaldson-Futaki invariant of the special degeneration and 
$A(F)-S(F)$ are same. Thus we can apply \cite[Theorem 1]{DS}. 
For \eqref{G_proposition2}, we may assume that there exists a $\Gamma$-invariant 
prime divisor $F$ over $X$ with $A(F)\leq S(F)$ such that the $\C$-algebra in 
\eqref{G_proposition1} is finitely generated. By Lemma \ref{alpha-tau_lemma} and 
\cite[Theorem 4.1]{alpha}, $X$ is isomorphic to $\pr^n$ and then we complete the proof.

\section{On the volume functions}\label{volume_section}

In this section, we generalize \cite[Proposition 2.1]{pltK} in order to show that 
$V_2$ admits K\"ahler-Einstein metrics. 

\begin{proposition}\label{vol_proposition}
Let $(X, \Delta)$ be an $n$-dimensional log Fano pair, let $F$ be a prime divisor over $X$, 
and let $0<a<b$ be positive real numbers. Assume that 
\[
\vol(-(K_X+\Delta)-xF)=\left(\frac{b-x}{b-a}\right)^n\vol(-(K_X+\Delta)-aF)
\]
for any $x\in[a,b]$. Then we have 
\[
S(F)\leq\frac{(n-1)a+b}{n+1}.
\]
\end{proposition}

\begin{proof}
The proof looks similar to the argument in the proof of \cite[Theorem 1.2]{SZ}. 
From the assumption, we have $\tau(F)=b$. Set $V:=(-(K_X+\Delta))^{\cdot n}$. 
By \cite[Theorem A]{BFJ}, the function 
$\vol(-(K_X+\Delta)-xF)$ is $\sC^1$ over $x\in[0,b)$. 
Let us set 
\[
f(x):=-\frac{1}{n}\frac{d}{dx}\vol(-(K_X+\Delta)-xF)
\]
as in \cite[Proof of Proposition 2.1]{pltK}. (We note that $f(x)$ is a restricted volume 
function in the sense of \cite{ELMNP}.) Then, for any $x\in[a,b)$, we have 
\[
f(x)=\left(\frac{b-x}{b-a}\right)^{n-1}f(a).
\]
As in \cite[Proof of Proposition 2.1]{pltK}, we have 
\begin{eqnarray*}
V&=&n\int_0^bf(x)dx, \\
S(F)&=&\frac{1}{V}\cdot n\int_0^bxf(x)dx, \\
f(x)&\geq&\left(\frac{x}{a}\right)^{n-1}f(a)\,\,\, \text{for any }x\in[0,a].
\end{eqnarray*}
In particular, we get 
\begin{eqnarray*}
V&\geq&n\left(\int_0^a\left(\frac{x}{a}\right)^{n-1}f(a)dx+
\int_a^b\left(\frac{b-x}{b-a}\right)^{n-1}f(a)dx\right)\\
&=&b\cdot f(a).
\end{eqnarray*}
Set 
\[
g(x):=\begin{cases}
\left(\frac{x}{a}\right)^{n-1}\cdot\frac{V}{b} & \text{ for }x\in[0,a], \\
\left(\frac{b-x}{b-a}\right)^{n-1}\cdot\frac{V}{b} & \text{ for }x\in[a,b].
\end{cases}
\]
Then we have 
\begin{eqnarray*}
V&=&n\int_0^bg(x)dx, \\
g(x)&\geq&f(x)\,\,\,\text{ for any }x\in[a,b].
\end{eqnarray*}
Moreover, the function 
\[
f(x)^{\frac{1}{n-1}}-g(x)^{\frac{1}{n-1}}=f(x)^{\frac{1}{n-1}}
-\frac{x}{a}\cdot\left(\frac{V}{b}\right)^{\frac{1}{n-1}}
\]
is $\sC^0$ and concave over $x\in[0,a]$ (by \cite[Theorem A]{ELMNP}). Moreover, 
we have $f(0)^{\frac{1}{n-1}}-g(0)^{\frac{1}{n-1}}\geq 0$ and 
$f(a)^{\frac{1}{n-1}}-g(a)^{\frac{1}{n-1}}\leq 0$. Thus there exists $c\in [0,a]$ such that 
\begin{eqnarray*}
f(x)&\geq&g(x)\quad\text{for any }x\in[0,c], \\
g(x)&\geq&f(x)\quad\text{for any }x\in[c,a]. 
\end{eqnarray*}
Thus we get 
\begin{eqnarray*}
&&n\int_0^bxf(x)dx-cV=n\int_0^b(x-c)f(x)dx\\
&\leq&n\int_0^b(x-c)g(x)dx=n\int_0^bxg(x)dx-cV.
\end{eqnarray*}
Since 
\[
n\int_0^bxg(x)dx=\frac{(n-1)a+b}{n+1}V, 
\]
we get the assertion. 
\end{proof}

\begin{proposition}\label{pltK_proposition}
Let $(X, \Delta)$ be an $n$-dimensional log Fano pair, let $\eta\in X$ be a 
scheme-theoretic point, and let $0<t\leq s$ be positive real numbers. Assume that 
there exists a prime divisor $T$ on $X$ with $T\sim_\Q-k(K_X+\Delta)$ for some 
$k\in\Q_{>0}$ such that 
\begin{itemize}
\item
the pair $(X, \Delta+\frac{t}{k}T)$ is lc at $\eta$, and 
\item
for any effective $\Q$-divisor $D'\sim_\Q-(K_X+\Delta)$ with $T\not\subset\Supp D'$, 
the pair $(X, \Delta+sD')$ is lc at $\eta$. 
\end{itemize}
Then, for any prime divisor $F$ over $X$ with $\eta\in c_X(F)$, we have the following: 
\begin{enumerate}
\renewcommand{\theenumi}{\arabic{enumi}}
\renewcommand{\labelenumi}{(\theenumi)}
\item\label{pltK_proposition1}
If $s^{-1}A(F)\leq\frac{1}{k}\ord_F T$, then we have 
\begin{eqnarray*}
&&\vol(-(K_X+\Delta)-xF)\\
&=&\left(\frac{\frac{1}{k}\ord_F T-x}{\frac{1}{k}\ord_F T-s^{-1}A(F)}\right)^n
\vol\left(-(K_X+\Delta)-s^{-1}A(F)F\right)
\end{eqnarray*}
for any 
$x\in[s^{-1}A(F), \frac{1}{k}\ord_F T]$.
\item\label{pltK_proposition2}
We have 
\[
S(F)\leq\frac{A(F)}{n+1}\left((n-1)s^{-1}+t^{-1}\right).
\]
\end{enumerate}
\end{proposition}

\begin{proof}
\eqref{pltK_proposition1}
Let us fix a log resolution $\pi\colon\tilde{X}\to X$ of $(X, \Delta)$ with 
$F\subset\tilde{X}$. 
Take any effective $\Q$-divisor $D\sim_\Q-(K_X+\Delta)$. Then there uniquely exists 
$e\in[0,1]\cap\Q$ such that we can write $D=\frac{e}{k}T+(1-e)D'$ with 
$D'\sim_\Q-(K_X+\Delta)$ effective and $T\not\subset\Supp D'$. 
Since $(X, \Delta+sD')$ is lc at $\eta$, we have $\ord_F D'\leq s^{-1}A(F)$. 
Assume that $x\in(s^{-1}A(F),\frac{1}{k}\ord_F T]\cap\Q$ satisfies that $\ord_F D\geq x$. 
(In other words, $mD\in|-m(K_X+\Delta)-mxF|$ holds for a sufficiently divisible $m\in\Z_{>0}$.) 
Then we have 
\[
e\geq\frac{x-s^{-1}A(F)}{\frac{1}{k}\ord_F T-s^{-1}A(F)},
\]
since 
\[
x\leq\ord_F D\leq e\cdot\frac{1}{k}\ord_F T+(1-e)s^{-1}A(F).
\]
This implies that the linear system $|-m(K_X+\Delta)-mxF|$ has a fixed divisor 
\[
m\cdot\frac{x-s^{-1}A(F)}{\frac{1}{k}\ord_F T-s^{-1}A(F)}\cdot\frac{1}{k}T
\]
for $m\in\Z_{>0}$ sufficiently divisible. Thus we get 
\begin{eqnarray*}
&&\vol(-(K_X+\Delta)-xF)\\
&=&\vol_{\tilde{X}}\left(-\pi^*(K_X+\Delta)-xF-
\frac{x-s^{-1}A(F)}{\frac{1}{k}\ord_F T-s^{-1}A(F)}\frac{1}{k}(\pi^*T-(\ord_F T)F)\right)\\
&=&\left(\frac{\frac{1}{k}\ord_F T-x}{\frac{1}{k}\ord_F T-s^{-1}A(F)}\right)^n
\vol\left(-(K_X+\Delta)-s^{-1}A(F)F\right).
\end{eqnarray*}

\eqref{pltK_proposition2}
Since the pair $(X, \Delta+\frac{t}{k}T)$ is lc at $\eta$, we have 
$\frac{1}{k}\ord_F T\leq t^{-1}A(F)$. 
If $\frac{1}{k}\ord_F T\leq s^{-1}A(F)$, then we have 
$\ord_F D\leq s^{-1}A(F)$ for any effective $\Q$-divisor $D\sim_\Q-(K_X+\Delta)$. 
Thus we get the inequality 
$s\leq\frac{A(F)}{\tau(F)}$. By \cite[Proposition 2.1]{pltK}, 
we get 
\begin{eqnarray*}
S(F)&\leq&\frac{n}{n+1}\tau(F)\leq\frac{n}{n+1}s^{-1}A(F)\\
&\leq&\frac{A(F)}{n+1}
\left((n-1)s^{-1}+t^{-1}\right).
\end{eqnarray*}
Thus we may assume that $\frac{1}{k}\ord_F T> s^{-1}A(F)$. In this case, we can apply (1). 
We have 
\begin{eqnarray*}
S(F)&\leq&\frac{1}{n+1}\left((n-1)s^{-1}A(F)+\frac{1}{k}\ord_F T\right)\\
&\leq&\frac{A(F)}{n+1}\left((n-1)s^{-1}+t^{-1}\right)
\end{eqnarray*}
by Proposition \ref{vol_proposition}. 
\end{proof}

\section{On quasi-log schemes}\label{q_section}

The theory of \emph{quasi-log schemes} plays an important role for the study of 
birational geometry. In this section, we see 
a kind of subadjunction theorem for projective qlc strata on quasi-log schemes, which is 
a direct consequence of the recent work \cite{fujino_hyp}. 
For the theory of quasi-log schemes, we refer the readers to \cite{fujino}. 

An $\R$-Weil divisor $D$ on a normal projective variety is said to be 
\emph{pseudo-effective} in this paper if $D+A$ is big (i.e., there exists an effective 
$\R$-Weil divisor $E$ such that $D+A-E$ is an ample $\R$-divisor) for any 
ample $\R$-divisor $A$.

\begin{thm}[{see \cite[Lemma 4.17, Theorems 1.9 
and 7.1]{fujino_hyp}}]\label{subadjunction_thm}
Let $[X, \omega]$ be a quasi-log scheme, let $C\subset X$ be a qlc stratum 
of $[X, \omega]$, and let $\nu\colon \bar{C}\to C$ be the normalization of $C$. 
Assume that $C$ is a projective variety. 
Then $\nu^*(\omega|_C)-K_{\bar{C}}$ is a pseudo-effective $\R$-Weil 
divisor on $\bar{C}$. 
\end{thm}

\begin{proof}
By \cite[Lemma 4.19]{fujino_hyp}, we may assume that $X=C$. Moreover, 
by \cite[Theorem 1.9]{fujino_hyp}, we may further assume that $X=C$ is normal. 
By \cite[Theorem 7.1]{fujino_hyp}, there exists a projective birational morphism 
$p\colon X'\to X$ from a smooth projective variety $X'$ such that we can write 
\[
K_{X'}+B_{X'}+M_{X'}=p^*\omega,
\]
where $B_{X'}$ is an effective $\R$-Weil divisor on $X'$ with $B_{X'}^{<0}$ $p$-exceptional, 
and $M_{X'}$ is a nef $\R$-divisor on $X'$. This immediately implies that 
\[
\omega-K_X=p_*\left(B_{X'}+M_{X'}\right)
\]
is a pseudo-effective $\R$-Weil divisor on $X$. 
\end{proof}

As a corollary, we get the following result, which is important for the proof of 
Theorem \ref{main_thm}. 

\begin{corollary}\label{qlc_corollary}
Let $(X, \Delta)$ be a log Fano pair, let $D\sim_\Q-(K_X+\Delta)$ be an effective 
$\Q$-divisor, and let $\alpha\in(0,1)\cap\Q$. 
Assume that the pair $(X, \Delta+\alpha D)$ is not klt. Let $\Nklt(X, \Delta+\alpha D)$ 
be the locus of non-klt points of $(X, \Delta+\alpha D)$. 
\begin{enumerate}
\renewcommand{\theenumi}{\arabic{enumi}}
\renewcommand{\labelenumi}{(\theenumi)}
\item\label{qlc_corollary1} 
The locus $\Nklt(X, \Delta+\alpha D)$ is connected. 
\item\label{qlc_corollary2} 
Take any $1$-dimensional irreducible component $B\subset\Nklt(X, \Delta+\alpha D)$ 
with its reduced scheme structure. Then $B$ is a rational curve with 
\[
\left(-(K_X+\Delta)\cdot B\right)\leq\frac{2}{1-\alpha}.
\]
Moreover, 
if any irreducible component of $\Nklt(X, \Delta+\alpha D)$ is of dimension 
$\leq 1$, then 
$B\simeq\pr^1$ and 
the restriction homomorphism 
\[
H^0(X, L)\to H^0(B, L|_B)
\]
is surjective for any nef line bundle $L$ on $X$.
\end{enumerate}
\end{corollary}

\begin{proof}
By \cite[6.4.1]{fujino}, the pair $(X, \Delta+\alpha D)$ admits a quasi-log structure 
$[X, \omega]$ with 
$\omega=K_X+\Delta+\alpha D$, and 
$N:=\Nklt(X, \Delta+\alpha D)$ has a natural scheme structure with 
\[
N=\bigcup_C C\cup X_{-\infty},
\]
where $C$ are the lc centers of $(X, \Delta+\alpha D)$ and $X_{-\infty}$
is the non-qlc locus of $(X, \Delta+\alpha D)$. 

For any nef line bundle $L$ on $X$, since 
\[
L-(K_X+\Delta+\alpha D)\sim_\Q L+(1-\alpha)(-K_X-\Delta)
\]
is ample, we have 
\[
H^i(X, L\otimes I_N)=0
\]
for any $i>0$ by \cite[Theorem 6.3.5 (ii)]{fujino}, where $I_N\subset\sO_X$ is the defining 
ideal sheaf of $N\subset X$. 

\eqref{qlc_corollary1}
For $L:=\sO_X$, we get the surjection 
\[
H^0(X, \sO_X)\twoheadrightarrow H^0(N, \sO_N).
\]
Thus $N$ is connected. 

\eqref{qlc_corollary2}
After replacing $\alpha$ with the log canonical threshold of $(X, \Delta; D)$ at the 
generic point of $B$, we may assume that $B$ is an lc center of $(X, \Delta+\alpha D)$. 
Take the normalization $\nu\colon \bar{B}\to B$. 
By Theorem \ref{subadjunction_thm}, the $\Q$-divisor 
\[
\nu^*\left((K_X+\Delta+\alpha D)|_B\right)-K_{\bar{B}}
\]
is pseudo-effective. 
This implies that $\bar{B}\simeq\pr^1$ and 
\begin{eqnarray*}
0&\leq&\deg_{\bar{B}}\left(\nu^*\left((K_X+\Delta+\alpha D)|_B\right)-K_{\bar{B}}\right)\\
&=&
-(1-\alpha)\left(-(K_X+\Delta)\cdot B\right)+2.
\end{eqnarray*}

Now we assume that $\dim N\leq 1$. 
Since $H^1(X, \sO_X)=H^2(X, I_N)=0$, we get $H^1(N, \sO_N)=0$. 
From the assumption $\dim N\leq 1$, we get $H^1(B, \sO_B)=0$, i.e., $B\simeq\pr^1$. 
Moreover, by \cite[Lemma 4.13 and the proof of Lemma 4.50]{KoMo}, 
the restriction homomorphism 
\[
H^0(N, L|_N)\to H^0(B, L|_B)
\]
is surjective for any nef line bundle $L$ on $X$. Since $H^1(X, L\otimes I_N)=0$, 
\[
H^0(X,L)\to H^0(B, L|_B)
\]
is also surjective. 
Thus we get the assertions.
\end{proof}

\section{Proof of Theorem \ref{main_thm}}\label{proof_section}

By Theorem \ref{CS_thm} and 
Proposition \ref{G_proposition} \eqref{G_proposition2}, 
we may assume that $u=2$. Set $X:=V_2$. Take any $G$-invariant prime divisor $F$ 
over $X$. (Recall that $G:=\C^*\rtimes(\Z/2\Z)$.)
By Proposition \ref{G_proposition} \eqref{G_proposition1}, it is enough to show 
the inequality $A(F)>S(F)$. Set $C:=c_X(F)$ and let $\eta\in X$ be the generic point of $C$. 
By Lemma \ref{alpha-tau_lemma} and \cite[Proposition 2.1]{pltK}, we may assume that 
$\alpha_{G,\eta}(X)\leq\frac{3}{4}$. Note that $C$ is a $G$-invariant subvariety on $X$. 
Thus $C$ is not a closed point by \cite[Lemma 2.23]{CS}. Moreover, if $F$ is a prime 
divisor on $X$, then $A(F)>S(F)$ holds by \cite[Corollary 9.3]{div_stability}. Thus 
we may assume that $C$ is a $G$-invariant curve. 
By Lemma \ref{solvable_lemma}, there exists a $G$-invariant effective $\Q$-divisor 
$D\sim_\Q-K_X$ and there exists $\alpha\in[\alpha_{G,\eta}(X),4/5)\cap\Q$ such that 
the pair $(X, \alpha D)$ is lc but not klt at $\eta$. We note that the non-klt locus 
of the pair $(X, \alpha D)$ is of dimension $\leq 1$ since $\Pic(X)=\Z[-K_X]$ and 
$\alpha<1$. 

By Corollary \ref{qlc_corollary}, $C$ is a smooth rational curve with 
\[
(-K_X\cdot C)\leq \frac{2}{1-\alpha}<10. 
\]
Moreover, 
for any $m\in\Z_{\geq 0}$, the restriction homomorphism 
\[
H^0(X, \sO_X(-mK_X))\to H^0(C, \sO_X(-mK_X)|_{C})
\]
is surjective. Thus $C\subset X\subset\pr^{13}=\pr^*H^0(X, -K_X)$ is a $G$-invariant 
rational normal curve in $\pr^{13}$ with $\deg C<10$. 

By \cite[Proposition 4.12, Lemma 7.7 and Corollary 7.10]{CS} and the assumption 
$\alpha_{G, \eta}(X)\leq\frac{3}{4}$, we may assume that $C=\sC_4$, where 
$\sC_4\subset X$ is the unique $G$-invariant rational normal curve of anti-canonical 
degree $4$ in $X$ (see \cite{CS} for the definition of the curve $\sC_4$). By 
\cite[Lemma 5.2 and the proof of Lemma 7.14]{CS}, there exists a prime divisor 
$T'_{15}\sim -K_X$ on $X$ such that 
\begin{itemize}
\item
the pair $(X, \frac{2}{3}T'_{15})$ is lc at $\eta$, and 
\item
the pair $(X, D')$ is lc at $\eta$ for any effective $\Q$-divisor $D'\sim_\Q -K_X$ 
with $T'_{15}\not\subset\Supp D'$. 
\end{itemize}
By Proposition \ref{pltK_proposition} \eqref{pltK_proposition2}, we get 
\[
S(F)\leq \frac{A(F)}{4}\left(2\cdot 1^{-1}+\left(\frac{2}{3}\right)^{-1}\right)=\frac{7}{8}A(F). 
\]
Thus we get the desired inequality $A(F)>S(F)$ and we complete the proof of 
Theorem \ref{main_thm}.


\begin{thebibliography}{99}

\bibitem[ACCFKMGSSV]{FANO}
C.\ Araujo, A-M.\ Castravet, I.\ Cheltsov, K.\ Fujita, A-S.\ Kaloghiros, 
J.\ Martinez-Garcia, C.\ Shramov, H.\ S\"uss and N.\ Viswanathan, 
\emph{The Calabi problem for Fano threefolds}, MPIM preprint 2021-31.

\bibitem[BFJ09]{BFJ}
S.\ Boucksom, C.\ Favre and M.\ Jonsson, \emph{Differentiability of volumes of 
divisors and a problem of Teissier}, 
J.\ Algebraic Geom.\ \textbf{18} (2009), no.\ 2, 279--308. 

\bibitem[BS64]{borel-serre}
A.\ Borel and J.-P.\ Serre, \emph{Th\'eor\`emes de finitude en cohomologie galoisienne}, 
Comment.\ Math.\ Helv.\ \textbf{39} (1964), 111-164. 

\bibitem[CS18]{CS}
I.\ Cheltsov and C.\ Shramov, \emph{K\"ahler-Einstein Fano threefolds of degree $22$}, 
arXiv:1803.02774v4.

\bibitem[Dem08]{Dem08}
J.-P.\ Demailly, Appendix to I.\ Cheltsov and C.\ Shramov's article. 
``Log canonical thresholds of smooth Fano threefolds” : 
\emph{On Tian's invariant and log canonical thresholds}, 
Russian Math.\ Surveys \textbf{63} (2008), no.\ 5, 945--950. 

\bibitem[Don02]{don}
S.\ Donaldson, \emph{Scalar curvature and stability of toric varieties}, 
J.\ Differential Geom.\ \textbf{62} (2002), no.\ 2, 289--349. 

\bibitem[Don07]{don-alpha}
S.\ Donaldson, \emph{A note on the $\alpha$-invariant of the Mukai-Umemura $3$-fold}, 
arXiv:0711.4357v1.

\bibitem[DS16]{DS}
V.\ Datar and G.\ Sz\'ekelyhidi, \emph{K\"ahler-Einstein metrics along the smooth 
continuity method}, Geom.\ Funct.\ Anal.\ \textbf{26}, no.\ 4, 975--1010. 

\bibitem[ELMNP09]{ELMNP}
L.\ Ein, R.\ Lazarsfeld, M.\ Musta\c{t}\u{a}, M.\ Nakamaye and 
M.\ Popa, \emph{Restricted volumes and base loci of linear series}, 
Amer.\ J.\ Math.\ \textbf{131} (2009), no.\ 3, 607--651. 

\bibitem[Fjn17]{fujino}
O.\ Fujino, \emph{Foundations of the minimal model program}, MSJ Memoirs, \textbf{35}. 
Mathematical Society of Japan, Tokyo, 2017. 

\bibitem[Fjn21]{fujino_hyp}
O.\ Fujino, \emph{Cone theorem and Mori hyperbolicity}, arXiv::2102.11986v1. 

\bibitem[Fjt16]{div_stability}
K.\ Fujita, \emph{On K-stability and the volume functions of 
$\mathbb{Q}$-Fano varieties}, Proc.\ Lond.\ Math.\ Soc.\ \textbf{113} (2016), no.\ 5, 
541--582.

\bibitem[Fjt19a]{vst}
K.\ Fujita, \emph{A valuative criterion for uniform K-stability of $\mathbb{Q}$-Fano 
varieties}, 
J.\ Reine Angew.\ Math.\ \textbf{751} (2019), 309--338.

\bibitem[Fjt19b]{pltK}
K.\ Fujita, \emph{Uniform K-stability and plt blowups of log Fano pairs}, 
Kyoto J.\ Math.\ \textbf{59} (2019), no.\ 2, 399--418. 

\bibitem[Fjt19c]{alpha}
K.\ Fujita, \emph{K-stability of Fano manifolds with not small alpha invariants}, 
J.\ Inst.\ Math.\ Jussieu \textbf{18} (2019), no.\ 3, 519--530. 

\bibitem[Hum75]{humphreys}
J.\ Humphreys, \emph{Linear algebraic groups}, Graduate Texts in Mathematics, No.\ 
\textbf{21}. Springer-Verlag, New York-Heidelberg, 1975. 

\bibitem[KM98]{KoMo}
J.\ Koll{\'a}r and S.\ Mori, \emph{Birational geometry of algebraic varieties},
With the collaboration of C.\ H.\ Clemens and A.\ Corti. 
Cambridge Tracts in Math., \textbf{134},
Cambridge University Press, Cambridge, 1998.

\bibitem[KP18]{KP}
A.\ Kuznetsov and Y.\ Prokhorov, \emph{Prime Fano threefolds of genus $12$ with a 
$\G_m$-action and their automorphisms}, \'Epijournal Geom.\ Alg\'ebrique \textbf{2} 
(2018), Art.\ 3, 14 pp. 

\bibitem[KPS18]{KPS}
A.\ Kuznetsov, Y.\ Prokhorov and C.\ Shramov, \emph{Hilbert schemes of lines and 
conics and automorphism groups of Fano threefolds}, Jpn.\ J.\ Math.\ \textbf{13}
(2018), no.\ 1, 109--185. 

\bibitem[Laz04a]{L1}
R.\ Lazarsfeld, \emph{Positivity in algebraic geometry, I: Classical setting: line bundles 
and linear series}, Ergebnisse der Mathematik und ihrer Grenzgebiete.\ (3) 
\textbf{48}, Springer, Berlin, 2004.

\bibitem[Laz04b]{L2}
R.\ Lazarsfeld, \emph{Positivity in algebraic geometry, II: Positivity for Vector Bundles, 
and Multiplier Ideals}, Ergebnisse der Mathematik und ihrer Grenzgebiete.\ (3) 
\textbf{49}, Springer, Berlin, 2004.

\bibitem[Li17]{li}
C.\ Li, \emph{K-semistability is equivariant volume minimization}, 
Duke Math.\ J.\ \textbf{166} (2017), no.\ 16, 3147--3218.

\bibitem[Mat57]{matsushima}
Y.\ Matsushima, \emph{Sur la structure du groupe d'hom\'eomorphismes analytiques 
d'une certaine vari\'et\'e k\"ahl\'erienne}, 
Nagoya Math.\ J.\ \textbf{11} (1957), 145--150.

\bibitem[MU83]{MU}
S.\ Mukai and H.\ Umemura, \emph{Minimal rational threefolds}, Algebraic geometry 
(Tokyo/Kyoto, 1982), 490--518, Lecture Notes in Math., \textbf{1016}, Springer, Berlin, 
1983. 

\bibitem[Pro90]{Prok}
Y.\ Prokhorov, \emph{Automorphism groups of Fano $3$-folds}, Uspekhi Mat.\ Nauk 
\textbf{45} (1990), no.\ 3 (273), 195--196. 

\bibitem[SZ19]{SZ}
C.\ Stibitz and Z.\ Zhuang, \emph{K-stability of birationally superrigid Fano varieties}, 
Compos.\ Math.\ \textbf{155} (2019), 1845--1852. 

\bibitem[Tia87]{Tia87}
G.\ Tian, 
\emph{On K\"ahler-Einstein metrics on certain K\"ahler manifolds with 
$C_1(M)>0$}, Invent.\ Math.\ \textbf{89} (1987), no.\ 2, 225--246.

\bibitem[Tia97]{tian}
G.\ Tian, \emph{K\"ahler-Einstein metrics with positive scalar curvature}, 
Invent.\ Math.\ \textbf{130} (1997), no.\ 1, 1--37.

\bibitem[Zhu20]{zhuang}
Z.\ Zhuang, \emph{Optimal destabilizing centers and equivariant K-stability}, 
arXiv:2004.09413v1. 
\end{thebibliography}
\end{document}